\documentclass[11pt]{amsart}
\usepackage[usenames,dvipsnames]{xcolor}
\usepackage{comment}
\usepackage{subcaption, wrapfig}
\usepackage{amsmath, amsfonts, amssymb, amsthm, geometry, graphics, graphicx, multicol, multirow, enumerate, float, floatflt, fullpage, tikz, outlines, url, tabularx}

\usetikzlibrary{shapes.multipart, matrix, positioning, arrows,cd, decorations.pathmorphing, decorations.pathreplacing}

\usepackage{stmaryrd}
\usepackage[utf8]{inputenc}
\newlength\bshft
\def\fakebold#1{\ThisStyle{\ooalign{$\SavedStyle#1$\cr%
  \kern-\bshft$\SavedStyle#1$\cr%
  \kern\bshft$\SavedStyle#1$}}}

\newtheorem{thm}{Theorem}[section]
\newtheorem{cor}[thm]{Corollary}
\newtheorem{defn}[thm]{Definition}
\newtheorem{prop}[thm]{Proposition}

\newtheorem{conj}[thm]{Conjecture}

\newtheorem{exa}[thm]{Example}
\newtheorem{obs}[thm]{Observation}
\newtheorem*{theorem*}{Theorem \ref{TorusLinksOnly2Torsion}}
\newtheorem*{theorem'}{Theorem \ref{nooddtorsion}}
\newtheorem*{theorem"}{Theorem \ref{GeneralTheorem}}
\newcommand{\R}{\mathbb{R}}
\newcommand{\Z}{\mathbb{Z}}
\newcommand{\Q}{\mathbb{Q}}

\newcommand{\C}{\mathbb{C}}

\newcommand{\MC}{\textnormal{MC}}

\newcommand{\rk}{\textnormal{rk }}
\newcommand{\rank}{\textnormal{rank}}

\newcolumntype{P}[1]{>{\centering\arraybackslash}p{#1}} %centered tabular columns of specified width
\definecolor{forest}{rgb}{0.03, 0.47, 0.19}
\newsavebox\squaregraph
\begin{lrbox}{\squaregraph}
\begin{tikzpicture}

\draw[fill=black] (1,0) circle (.7mm);
\draw[fill=black] (0,1) circle (.7mm);
\draw[fill=black] (1,1) circle (.7mm);
\draw[fill=black] (0,0) circle (.7mm);

\draw (0,0)--(1,0);
\draw (1,0)--(1,1);
\draw (1,1)--(0,1);
\draw (0,1)--(0,0);

\end{tikzpicture}
\end{lrbox}

\newsavebox\twosimplex
\begin{lrbox}{\twosimplex}

\begin{tikzpicture}

\draw[fill=black!25] (1,1.25)--(3,1.25)--(2,2.98)--cycle;
\node at (.8,1.05){$0$};
\node at (3.2,1.05){$1$};
\node at (2,3.25){$2$};

\end{tikzpicture}

\end{lrbox}

\newsavebox\YKtwosimplex
\begin{lrbox}{\YKtwosimplex}
\begin{tikzpicture}

\draw[fill=black] (9,0) circle (.7mm);
\draw[fill=black] (7,1.5) circle (.7mm);
\draw[fill=black] (9,1.5) circle (.7mm);
\draw[fill=black] (11,1.5) circle (.7mm);
\draw[fill=black] (7,3) circle (.7mm);
\draw[fill=black] (9,3) circle (.7mm);
\draw[fill=black] (11,3) circle (.7mm);
\draw[fill=black] (9,4.5) circle (.7mm);

\draw (9,0)--(7,1.5);
\draw (9,0)--(9,1.5);
\draw (9,0)--(11,1.5);
\draw (7,1.5)--(7,3);
\draw (7,1.5)--(9,3);
\draw (9,1.5)--(7,3);
\draw (9,1.5)--(11,3);
\draw (11,1.5)--(9,3);
\draw (11,1.5)--(11,3);
\draw (7,3)--(9,4.5);
\draw (9,3)--(9,4.5);
\draw (11,3)--(9,4.5);

\node[scale=.9] at (9,-.4){$\hat{0}$};
\node[scale=.7] at (6.5,1.5){$[0]$};
\node[scale=.7] at (8.5,1.5){$[1]$};
\node[scale=.7] at (10.5,1.5){$[2]$};
\node[scale=.7] at (6.35,3){$[0,1]$};
\node[scale=.7] at (8.35,3){$[0,2]$};
\node[scale=.7] at (10.35,3){$[1,2]$};
\node[scale=.9] at (9,4.9){$\hat{1}$};

\end{tikzpicture}
\end{lrbox}
\title{Torsion in the Magnitude Homology of Graphs}
\author{Radmila Sazdanovic and Victor Summers}

\begin{document}

\maketitle

\begin{abstract}
Magnitude homology is a bigraded homology theory for finite graphs defined by Hepworth and Willerton, categorifying the power series invariant known as magnitude which was introduced by Leinster. We analyze the structure and implications of torsion in magnitude homology. 
We show that any finitely generated abelian group may appear as a subgroup of the magnitude homology of a graph, and, in particular, that torsion of a given prime order can appear in the magnitude homology of a graph and that there are infinitely many such graphs. Finally, we provide complete computations of magnitude homology of outerplanar graphs and focus on the ranks of the groups along the main diagonal of magnitude homology.
\end{abstract}

\section{Introduction}
Magnitude is an invariant of enriched categories introduced by Tom Leinster \cite{leinster2013magnitudematricspaces}. A finite graph may be viewed as a generalized metric space, which, in turn, may be viewed as an enriched category. The resulting magnitude invariant of finite graphs takes the form of a single-variable power series \cite{leinster2019magnitude}. Magnitude homology is a bigraded homology theory for finite graphs, introduced by Hepworth and Willerton \cite{hepworth2017categorifying}, that categorifies magnitude in the sense that its graded Euler characteristic is the magnitude invariant. Hepworth and Willerton initially conjectured that torsion is not to be found in magnitude homology groups, but this was subsequently shown to be false by Kaneta and Yoshinaga \cite{YoshinagaKaneta2018}. The structure of the magnitude homology and the existence and types of torsion are the primary focus of this paper.

We show that any finitely generated abelian group may appear as a subgroup of the magnitude homology of a graph, and, in particular, that torsion of a given prime order can appear in the magnitude homology of a graph and that there are infinitely many such graphs. Finally, we provide  complete computations of magnitude homology of outerplanar graphs and focus on the ranks of the groups along the main diagonal of magnitude homology building on and contributing to results in \cite{hepworth2017categorifying, hepworth2018magnitude, YoshinagaKaneta2018, gu2018graph} and others.

This paper is organized as follows. In Section \ref{Background} we recall the constructions of magnitude and magnitude homology along with their basic properties. In Section \ref{TorsionInMagnitudeHomology} we review the current state of knowledge of torsion in magnitude homology, show that torsion of arbitrary prime order can be found in magnitude homology (Theorem \ref{ptorsioninMH}), and construct infinite families of graphs with a given prime order in magnitude homology (Theorems \ref{InfinitelyManypTorsion} and \ref{InfinitelyMany2Torsion}). Finally, in Section \ref{ComputationsAndMainDiagonal} we extend Gu's computations \cite{gu2018graph} of the magnitude homology groups of cycle graphs to a family of outerplanar graphs (Theorem \ref{Evenouterplanarconjecture}), compute the main diagonal of graphs with no induced cycles of length $3$ or $4$ (Theorem \ref{nosquaresortriangles}). In Section \ref{lastsec} we put forth several conjectures on the structure of the main diagonal for other families of graphs based on calculations performed in Python code by R. Hepworth and S. Willerton.

\section*{Acknowledgements}
We thank Tye Lidman and Simon Willerton for helpful conversations.  RS was partially supported by the Simons Collaboration Grant 318086  and NSF DMS $1854705$.

\section{Background on magnitude and magnitude homology}\label{Background}

A graph is a pair $G=(V,E)$ where $V$ is a set of vertices and $E$ is a set of unordered pairs of vertices, which we think of as a set of undirected edges between vertices. For the purposes of defining magnitude and magnitude homology, we assume all graphs to have no loops and no double edges \cite{leinster2019magnitude}. Such a graph $G$ may be viewed as an extended metric space (a metric space with infinity allowed as a distance) whose points are the vertices of $G$ by declaring each edge to be of length one and defining an extended metric $d : V \times V \rightarrow [0,\infty]$ by setting $d(x,y) $ to be equal to the length of a shortest path in $G$ from $x$ to $y$:

 $ \min\left\{ d(x,x_1) + \displaystyle{\sum_{i=1}^{k-2} d(x_i,x_{i+1})} + d(x_{k-1},y): \{x,x_1\}, \{x_{k-1},y\}, \{x_i,x_{i+1}\} \in E \ \textnormal{for} \ 1 \le i \le k-2 \right\}$

 By convention we let $d(x,y)=\infty$ if there is no path from $x$ to $y$. For example, in the cycle graph $C_6$ shown in Figure \ref{C6K4} $d(a_1,a_5)=2$ and $d(a_0,a_3)=3$, while the distance between any two distinct vertices of complete graph $K_4$ is $1$.

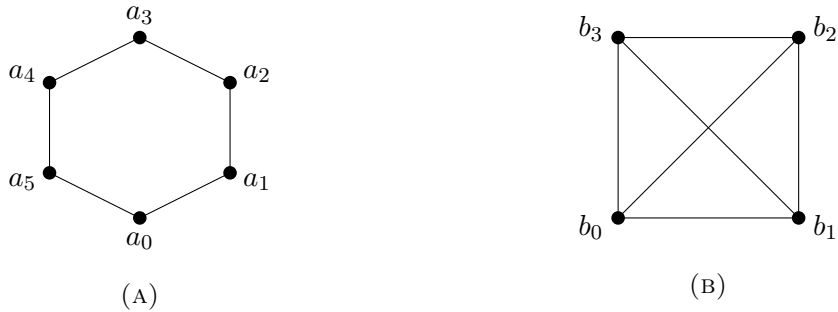
\begin{figure}[H]
\begin{subfigure}{0.45\textwidth}
\centering
\begin{center}
\begin{tikzpicture}[scale=.6]
%\draw[help lines] (0,0) grid (26,10);

%C6
\fill (0,2) circle (1.5mm);
\node at (-.6,1.8){$a_5$};
\fill (0,4) circle (1.5mm);
\node at (-.6,4.2){$a_4$};
\fill (2,1) circle (1.5mm);
\node at (2,.5){$a_0$};
\fill (4,2) circle (1.5mm);
\node at (4.6,1.8){$a_1$};
\fill (4,4) circle (1.5mm);
\node at (4.6,4.2){$a_2$};
\fill (2,5) circle (1.5mm);
\node at (2,5.5){$a_3$};
\draw (0,2)--(0,4);
\draw (0,4)--(2,5);
\draw (2,5)--(4,4);
\draw (4,4)--(4,2);
\draw (4,2)--(2,1);
\draw (2,1)--(0,2);
\end{tikzpicture}
\end{center}
\caption{}
\end{subfigure}
\begin{subfigure}{0.45\textwidth}
\begin{center}
\begin{tikzpicture}[scale=.6]
\fill (9,1) circle (1.5mm);
\node at (8.4,.8){$b_0$};
\fill (13,1) circle (1.5mm);
\node at (13.6,.8){$b_1$};
\fill (13,5) circle (1.5mm);
\node at (13.6,5.2){$b_2$};
\fill (9,5) circle (1.5mm);
\node at (8.4,5.2){$b_3$};
\draw (9,1)--(13,1);
\draw (13,1)--(13,5);
\draw (13,5)--(9,5);
\draw (9,5)--(9,1);
\draw (9,1)--(13,5);
\draw (9,5)--(13,1);
\end{tikzpicture}
\end{center}
\caption{}
\end{subfigure}
\caption{(a) The cycle graph $C_6$ on six vertices. (b) The complete graph $K_4 \ $on four vertices.}
    \label{C6K4}
\end{figure}

\begin{defn}[\cite{leinster2019magnitude}]
Let $G$ be a graph with (ordered) vertex set $V = \{ v_1,v_2,\ldots,v_n \}$. The similarity matrix of $G$ is the $n \times n$ square matrix with entries in the polynomial ring $\ \Z[q]$ given by
\[ Z_G(q) = 
\begin{bmatrix}
q^{d(v_1,v_1)} & q^{d(v_1,v_2)} & \dots & q^{d(v_1,v_n)}\\
q^{d(v_2,v_1)} & \ddots & & q^{d(v_2,v_n)}\\
\vdots & & & \vdots\\
q^{d(v_n,v_1)} & q^{d(v_n,v_2)} & \dots & q^{d(v_n,v_n)}
\end{bmatrix}
\]
where it is understood that $q^{\infty}=0$.
\end{defn}

Since $Z_G(0)\ $ is the identity matrix, the polynomial $\det(Z_G(q))$ has constant term $1$. Consequently, $\det(Z_G(q))$ is invertible in the ring $\Z\llbracket q \rrbracket$ of power series in the variable $q$. This allows for the following definition due to Leinster.

\begin{defn}[\cite{leinster2019magnitude}]
Let $G$ be a graph and $Z_G(q)$ its similarity matrix. The magnitude of a graph $G$, denoted by $\#G = \#G(q)$, is the sum of all entries of the inverse matrix $Z_G(q)^{-1}$.
\end{defn}

The magnitude invariant has many cardinality-like properties \cite{Leinster2008EulerChar}. For example, the magnitude of a graph $G$ with no edges is precisely its cardinality as a set: $\#G = |V(G)|$. Magnitude is also multiplicative with respect to Cartesian products and, under fairly mild conditions, also satisfies an inclusion-exclusion formula \cite{leinster2019magnitude}. Another angle from which cardinality-like properties can be seen is the magnitude function. The magnitude function $f_G(t)$ is the partially defined function of the extended real numbers obtained by setting $q=e^{-t}$. Although this function may have singularities, it is known to be increasing for large enough $t$ and satisfies $\displaystyle\lim_{t \rightarrow \infty} f_G(t) = |V(G)|$. These observations and their full details can be found in \cite{Leinster2013Asymptotic,leinster2019magnitude}.

Due to the alternating nature of the formula for Euler characteristic, an alternating sum formula for an invariant is often seen as a potential starting point for a categorification, i.e. can we build a chain complex whose graded Euler characteristic is the invariant in question? In the case of magnitude, it was the alternating sum formula of Proposition \ref{alternatingformula} that was the starting point of Hepworth and Willerton's construction of magnitude homology, a categorification of the magnitude power series.

\begin{prop}[\cite{leinster2019magnitude}, Proposition 3.9] \label{alternatingformula}
For any graph $G$,

\[
\#G(q) = \sum_{k=0}^{\infty} (-1)^k \sum _{x_0 \neq x_1 \neq \cdots \neq x_k} q^{d(x_0,x_1) + d(x_1,x_2) + \cdots + d(x_{k-1},x_k)}
\]

where the $x_i$ denote vertices of $G$. That is, if $\#G(q) = \sum_{n=0}^{\infty} c_{\ell} q^{\ell}$, then the $c_{\ell}$ are given by

\[
    c_{\ell} = \sum_{k=0}^{\ell} (-1)^k |  \{ (x_0,x_1,\ldots,x_k) : x_0 \neq x_1 \neq \cdots \neq x_k, \ d(x_0,x_1) + d(x_1,x_2) + \cdots + d(x_{k-1},x_k) = \ell \} |
\]
\end{prop}

Next, recall Hepworth and Willerton's construction \cite{hepworth2017categorifying} of the magnitude homology groups of a graph, along with some of its most notable properties.

\begin{defn}
A $k$-path in $G$ is a $(k+1)$-tuple $(x_0,x_1,\ldots,x_k)$ of vertices in $G$ with $x_i \neq x_{i+1}$ and $d(x_i,x_{i+1}) < \infty$ for each $0 \le i \le k-1$. The length of a $k$-path $(x_0,x_1,\ldots,x_k)$ in $G$ is
\[ \ell(x_0,x_1,\ldots,x_k) = d(x_0,x_1) + d(x_1,x_2) + \cdots d(x_{k-1},x_k). \]
\end{defn}

\begin{defn}[Magnitude chain complex, \cite{hepworth2017categorifying}]
Let $G$ be a graph. Let $\MC(G) = \bigoplus_{k,\ell \ge 0} \MC_{k,\ell}(G)$ be the bigraded $\Z$-module with components
\[
\MC_{k,\ell}(G) := \Z\{ \textbf{x} = (x_0,x_1,\ldots,x_k) : x_0 \neq x_1 \neq \ldots \neq x_k, \ \ell(\textbf{x})=\ell) \}
\]
That is, $\MC(G)$ is generated in bigrading $(k,\ell)$ by all $k$-paths in $G$ of length $\ell$. For $1 \le i \le k-1$, define maps $\partial_i : \MC_{k,\ell}(G) \rightarrow \MC_{k-1,\ell}(G)$ by
\[
\partial_i( x_0,x_1,\ldots,x_k) = \delta^{\ell,\ell(x_0,\ldots,x_{i-1},x_{i+1},\ldots,x_k)} (x_0,\ldots,x_{i-1},x_{i+1},\ldots,x_k)
\]
That is, $\partial_i$ removes vertex $x_i$ if the length of the path is preserved, and is the zero map otherwise. Defining maps $\partial : \MC_{k,\ell} \rightarrow \MC_{k-1,\ell}$ by $\partial = \sum_{i=1}^{k-1} (-1)^{i} \partial_i$
we find that $(\MC_{*,\ell},\partial)$ forms a chain complex for each $\ell \ge 0$. The magnitude homology of $G$ is the bigraded $\Z$-module  $\textnormal{MH}(G)$ given in bigrading $(k,\ell)$ by
\[
\textnormal{MH}_{k,\ell}(G) := H_k(\MC_{*,\ell}(G)).
\]
\end{defn}

Let us recall some basic properties of magnitude homology. First, magnitude homology categorifies the magnitude invariant in the sense that the magnitude of a graph $G$ may be recovered as the graded Euler characteristic of its magnitude homology:

\begin{align}
\chi_q(\MC(G))
&= \sum_{\ell \ge 0} \left( \sum_{k \ge 0} (-1)^k \ \rk(\textnormal{MH}_{k,\ell}(G)) \right)\cdot q^{\ell}\notag\\
&= \sum_{\ell \ge 0} \left( \sum_{k \ge 0} (-1)^k \ \rk(\MC_{k,\ell}(G)) \right)\cdot q^{\ell}\label{line2}\\
&= \#G(q)\label{line3}.
\end{align}

Line (\ref{line2}) holds because the differential preserves the second grading. i.e. is degree $0$, while line (\ref{line3}) follows by Proposition \ref{alternatingformula}. Second, magnitude homology is strictly stronger than magnitude; in \cite{gu2018graph}, Gu shows that the Rook(4,4) and Shrikhande graphs have the same magnitude but non-isomorphic magnitude homology groups. Third, the multiplicativity of magnitude with respect to Cartesian products of graphs lifts to a K\"{u}nneth sequence in magnitude homology \cite{hepworth2017categorifying}.

\begin{defn}[\cite{hepworth2017categorifying}]
The Cartesian product $G_1 \Box G_2$ of graphs $G_1$ and $G_1$ has vertex set $V(G_1) \times V(G_2)$ and an edge between vertices $(x_1,x_2)$ to $(y_1,y_2)$ if either $x_1=y_1$ and there is an edge in $G_2$ between $x_2$ and $y_2$ in $G_2$, or $x_2=y_2$ and there is an edge between $x_1$ and $y_1$ in $G_1$.
\end{defn}

\begin{thm}[\cite{hepworth2017categorifying}]
Let $G_1$ and $G_2$ be graphs. Then, magnitude satisfies the formula
\begin{equation}\label{Multiplicativity}
\#(G_1 \Box G_2) = \#G_1 \cdot \#G_2
\end{equation}
\end{thm}

\begin{thm}[K\"{u}nneth sequence, \cite{hepworth2017categorifying}]
Let $G_1$ and $G_2$ be graphs. Then, magnitude homology satisfies a short exact sequence of the form
\begin{equation}\label{Kunneth}
0 \rightarrow \textnormal{MH}_{*,*}(G_1) \otimes \textnormal{MH}_{*,*}(G_2) \rightarrow \textnormal{MH}_{*,*}(G_1 \Box G_2) \rightarrow \textnormal{Tor}_1^{\Z}(\textnormal{MH}_{*-1,*}(G_1),\textnormal{MH}_{*,*}(G_2)) \rightarrow 0.
\end{equation}
\end{thm}

The K\"{u}nneth sequence (\ref{Kunneth}) lifts equation (\ref{Multiplicativity}) in the sense that taking the graded Euler characteristic of the former yields the latter.

The fourth property noted here is a Mayer-Vietoris-type sequence for magnitude homology. This sequence relates the magnitude homology of a union of two subgraphs to the magnitude of the subgraphs and their intersection. As we will see, this sequence lifts the inclusion-exclusion formula for magnitude and holds for so-called projecting decompositions. The following definitions can be found in \cite{leinster2019magnitude} and \cite{hepworth2017categorifying}.

\begin{defn}[Convex subgraph, \cite{leinster2019magnitude}]
Let $H$ be a subgraph of a graph $G$, let $d$ be the shortest path metric on $G$, and let $d_H$ be the shortest path metric $H$. $H$ is said to be convex in $G$ if $d_H(x,y) = d(x,y)$
for vertices $x,y \in H$.
\end{defn}

\begin{defn}[Projecting subgraph, \cite{leinster2019magnitude}]
Let $H$ be a convex subgraph of a graph $G$, and write
\[
V_H(G) = \bigcup_{h \in H} \{ x \in G \ | \ d(x,h) < \infty \}.
\]
In other words, $V_H(G)$ consists of those vertices of $G$ in the same connected component of $H$. $G$ is said to project onto $H$ if for each $x \in V_H(G)$ there is a vertex $\pi(x) \in H$ such that for each $h \in H$,
\[
d(x,h) = d(x,\pi(x)) + d(\pi(x),h).
\]
\end{defn}

If a graph $G$ projects onto a convex subgraph $H$, then we get a projection map $\pi : V_H(G) \rightarrow V(H), \ x \mapsto \pi(x)$, sending each $x$ to its unique closest neighbor in $H$. This map is analogous to the projection map onto a closed, convex subset of Euclidean space.

\begin{exa}
For a positive integer $r \ge 2$, the cycle graph $C_r$ has vertex set $V = \Z_r$ and edge set $E = \{ \{ i, i+1 \} : i \in \Z_r \}$. Each even cycle graph $C_{2n}$ projects onto a single edge, but no odd cycle graph $C_{2n+1}$ projects onto a single edge.
\end{exa}

\begin{defn}[Projecting decomposition, \cite{hepworth2017categorifying}]
Let $H_1$ and $H_2$ be subgraphs of a graph $G$. The triple $(G;H_1,H_2)$ is said to be a projecting decomposition of $G$ if the following hold:

\[
1. \; G = H_1 \cup H_2 \qquad 2. \; H_1 \cap H_2 \ \textnormal{is convex in} \ G \qquad 3. \; H_1 \ \textnormal{projects onto} \ H_1 \cap H_2.
\]

\end{defn}

\begin{thm}[Inclusion-exclusion formula, \cite{leinster2019magnitude}]
Let $(G;H_1,H_2)$ be a projecting decomposition of a graph $G$. Then, magnitude satisfies $ \ \#G = \#H_1 + \#H_2 - \#(H_1 \cup H_2)$.
\end{thm}

This inclusion-exclusion formula lifts to a Mayer-Vietoris-type sequence in magnitude homology.

\begin{thm}[Mayer-Vietoris for magnitude homology, \cite{hepworth2017categorifying}]
Let $(G;H_1,H_2)$ be a projecting decomposition of a graph $G$. Then, magnitude homology satisfies a short exact sequence of the form
\begin{equation}\label{MVSequence}
0 \rightarrow \textnormal{MH}_{*,*}(H_1 \cap H_2) \rightarrow \textnormal{MH}_{*,*}(H_1) \oplus \textnormal{MH}_{*,*}(H_2) \rightarrow \textnormal{MH}_{*,*}(G) \rightarrow 0
\end{equation}
where the middle two maps are induced by inclusions. Moreover, the sequence splits.
\end{thm}

The Mayer-Vietoris sequence lifts the inclusion-exclusion formula in the sense that taking the graded Euler characteristic of the former yields the latter. The Mayer-Vietoris sequence (\ref{MVSequence}) is the essential computational tool of Section \ref{ComputationsAndMainDiagonal} when it comes to determining homology groups of a family of outerplanar graphs.

\section{Torsion in Magnitude Homology}\label{TorsionInMagnitudeHomology}

In this section, we review the current state of knowledge regarding torsion in magnitude homology, show that torsion of a prescribed prime order can be found in magnitude homology, and construct infinite families of graphs with a given order of torsion in magnitude homology.

\subsection{Torsion of prime order in magnitude homology}\label{ExistenceOfPTorsion}

In 2017, Kaneta and Yoshinaga demonstrated the existence of a graph with torsion of order two in magnitude homology \cite{YoshinagaKaneta2018}. More specifically, they describe a method for constructing a graph from a simplicial complex, such that the homology of the simplicial complex embeds into the magnitude homology of the resulting graph. Next, we describe this construction and extend their result on torsion of order two by constructing graphs whose magnitude homology contains torsion of a given order, and more generally a subgroup isomorphic to a given finitely generated abelian group.

\begin{defn}[Face poset, \cite{Wachs2006PosetsOrderComplexes}]
Let $K$ be a simplicial complex. The face poset of $K$ is the partially ordered set whose elements are the faces of $K$, ordered by inclusion. Denote this poset by $P(K)$.
\end{defn}

\begin{defn}[\cite{YoshinagaKaneta2018}]
Let $K$ be a simplicial complex of dimension $m$ such that each face of $K$ is the face of an $m$-simplex. Kaneta and Yoshinaga construct a graph $G(K)$ as follows. Let $\widehat{P(K)}$ be the poset obtained from $P(K)$ by adjoining a minimum element $\hat{0}$ (if it does not already have one) and a maximum element $\hat{1}$ (if it does not already have one). Then, $G(K)$ is the underlying graph of the Hasse diagram of $\widehat{P(K)}$.
\end{defn}

\begin{figure}[H]
\begin{center}
\begin{tikzpicture}[scale=.8]
%\draw[help lines] (0,0) grid (12,6);
\node at (2,2.25){\usebox\twosimplex};

\node at (2,2){\textcolor{red}{$K$}};

\draw[->] (4,2.25)--(6,2.25);

\node at (9.5,2.25){\usebox\YKtwosimplex};

\node at (10.1,2.1){\textcolor{red}{$G(K)$}};
\end{tikzpicture}
\end{center}
\caption{A minimal triangulation $K$ of a disc and the underlying graph $G(K)$ of the poset $\widehat{P(K)}$.}
\end{figure}
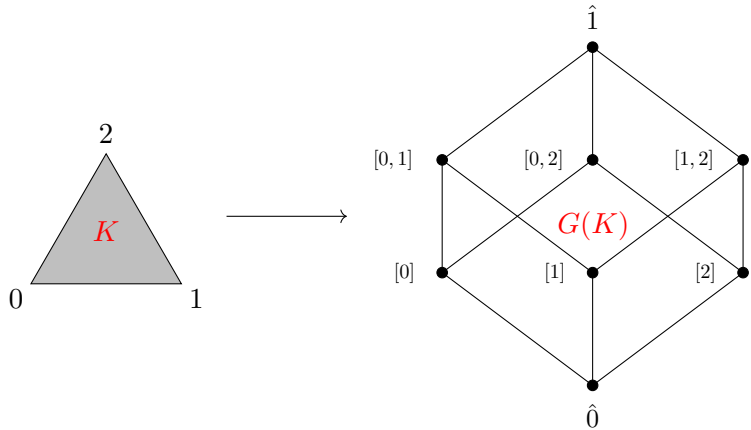

Recall that a triangulation of a topological space $X$ is a pair $(K,h)$ where $K$ is a simplicial complex and $h : |K| \rightarrow X$ is a homeomorphism of the geometric realization $|K|$ of $K$ with $X$. By a small abuse of notation, let us write $K=(K,h)$. The following lemma is our main tool for producing graphs with torsion in magnitude homology.

\begin{thm}[\cite{YoshinagaKaneta2018}, Corollary 5.12 (a)]\label{YKembeddinglemma}

Let $K$ be a triangulation of a manifold $M$ and let $\ell = d(\hat{0},\hat{1})$ be the distance between to the minimum and maximum elements of $\widehat{P(K)}$. For each $k \ge 1$ there is an embedding of the reduced singular homology groups of $M$ into the magnitude homology of the associated graph $G(K)$,

\begin{equation}\label{YKembeddingequation}
\widetilde{H}_{k-2}(M) \hookrightarrow \textnormal{MH}_{k,\ell}(G(K)).
\end{equation}
\end{thm}

\begin{comment}
\begin{obs}
Suppose $K$ is a triangulation of an $m$-dimensional manifold $M$. The poset $P(K)$ always has the empty simplex $\hat{0}=\emptyset$ as its minimum element. If $K$ has a single $m$-simplex, that is, if $K$ is a triangulation of the  $\Delta^m$, then $K$ has maximum element $\hat{1}=\Delta^m$. In this case, $d(\hat{0},\hat{1})=d(\emptyset,\Delta^m)=m+1$, and Theorem \ref{YKembeddinglemma} yields embeddings

\begin{equation}\label{YKembeddingdimM}
\widetilde{H}_{k-2}(M) \hookrightarrow \textnormal{MH}_{k,m+2}(G(K))
\end{equation}

On the other hand, if $K$ has more than one $m$-simplex, then the construction of $\widehat{P(K)}$ requires that we endow $P(K)$ with a maximum element. In this case, $d(\hat{0},\hat{1})=m+2$ and Theorem \ref{YKembeddinglemma} yields embeddings

\begin{equation}\label{dim=m+2}
\widetilde{H}_{k-2}(M) \hookrightarrow \textnormal{MH}_{k,m+1}(G(K)).
\end{equation}

\end{obs}
\end{comment}
\begin{exa}[\cite{YoshinagaKaneta2018}]\label{exYKRP2}
Let $K$ be a minimal triangulation of $\R\mathbb{P}^2$ as shown in Figure \ref{Fig:RP2TriangulationAndGraph} (a). By Theorem \ref{YKembeddinglemma} there is an embedding
$\Z_2 \cong H_1(\mathbb{R}\mathbb{P}^2) \hookrightarrow \textnormal{MH}_{3,4}(G(K))$.

\begin{figure}[H]
\begin{subfigure}{0.48\textwidth}
\centering
\begin{tikzpicture}[scale=.45]
%\draw[help lines] (0,0) grid (12,12);

\filldraw[draw=black, fill=red!15] (6,2)--(4,5)--(1,5)--cycle;
\filldraw[draw=black, fill=red!15] (6,2)--(8,5)--(4,5)--cycle;
\filldraw[draw=black, fill=red!15] (6,2)--(11,5)--(8,5)--cycle;
\filldraw[draw=black, fill=red!15] (1,5)--(4,5)--(1,9)--cycle;
\filldraw[draw=black, fill=red!15] (4,5)--(8,5)--(6,9)--cycle;
\filldraw[draw=black, fill=red!15] (8,5)--(11,5)--(11,9)--cycle;
\filldraw[draw=black, fill=red!15] (4,5)--(6,9)--(1,9)--cycle;
\filldraw[draw=black, fill=red!15] (8,5)--(11,9)--(6,9)--cycle;
\filldraw[draw=black, fill=red!15] (1,9)--(6,9)--(6,12)--cycle;
\filldraw[draw=black, fill=red!15] (6,9)--(11,9)--(6,12)--cycle;

\end{tikzpicture}
\caption{}
\end{subfigure}
\begin{subfigure}{0.48\textwidth}
\centering
\begin{tikzpicture}[scale=.35]
%\draw[help lines] (0,0) grid (18,12);
\fill (9,0) circle (1.5mm);

\draw (9,0)--(4,3);
\draw (9,0)--(6,3);
\draw (9,0)--(8,3);
\draw (9,0)--(10,3);
\draw (9,0)--(12,3);
\draw (9,0)--(14,3);

\fill (4,3) circle (1.5mm);
\fill (6,3) circle (1.5mm);
\fill (8,3) circle (1.5mm);
\fill (10,3) circle (1.5mm);
\fill (12,3) circle (1.5mm);
\fill (14,3) circle (1.5mm);

\draw (2,6)--(4,3);
\draw (3,6)--(4,3);
\draw (5,6)--(4,3);
\draw (8,6)--(4,3);
\draw (12,6)--(4,3);

\draw (4,6)--(6,3);
\draw (6,6)--(6,3);
\draw (9,6)--(6,3);
\draw (13,6)--(6,3);
\draw (2,6)--(6,3);

\draw (3,6)--(8,3);
\draw (4,6)--(8,3);
\draw (7,6)--(8,3);
\draw (10,6)--(8,3);
\draw (14,6)--(8,3);

\draw (5,6)--(10,3);
\draw (6,6)--(10,3);
\draw (7,6)--(10,3);
\draw (11,6)--(10,3);
\draw (15,6)--(10,3);

\draw (8,6)--(12,3);
\draw (9,6)--(12,3);
\draw (10,6)--(12,3);
\draw (11,6)--(12,3);
\draw (16,6)--(12,3);

\draw (16,6)--(14,3);
\draw (15,6)--(14,3);
\draw (14,6)--(14,3);
\draw (13,6)--(14,3);
\draw (12,6)--(14,3);

\draw (0,9)--(2,6);
\draw (8,9)--(2,6);

\draw (2,9)--(3,6);
\draw (4,9)--(3,6);

\draw (6,9)--(4,6);
\draw (10,9)--(4,6);

\draw (0,9)--(5,6);
\draw (2,9)--(5,6);

\draw (0,9)--(6,6);
\draw (14,9)--(6,6);

\draw (2,9)--(7,6);
\draw (16,9)--(7,6);

\draw (4,9)--(8,6);
\draw (12,9)--(8,6);

\draw (6,9)--(9,6);
\draw (14,9)--(9,6);

\draw (4,9)--(10,6);
\draw (6,9)--(10,6);

\draw (14,9)--(11,6);
\draw (18,9)--(11,6);

\draw (8,9)--(12,6);
\draw (12,9)--(12,6);

\draw (8,9)--(13,6);
\draw (10,9)--(13,6);

\draw (10,9)--(14,6);
\draw (16,9)--(14,6);

\draw (16,9)--(15,6);
\draw (18,9)--(15,6);

\draw (12,9)--(16,6);
\draw (18,9)--(16,6);

\draw (0,9)--(9,12);
\draw (2,9)--(9,12);
\draw (4,9)--(9,12);
\draw (6,9)--(9,12);
\draw (8,9)--(9,12);
\draw (10,9)--(9,12);
\draw (12,9)--(9,12);
\draw (14,9)--(9,12);
\draw (16,9)--(9,12);
\draw (18,9)--(9,12);

\fill (2,6) circle (1.5mm);
\fill (3,6) circle (1.5mm);
\fill (4,6) circle (1.5mm);
\fill (5,6) circle (1.5mm);
\fill (6,6) circle (1.5mm);
\fill (7,6) circle (1.5mm);
\fill (8,6) circle (1.5mm);
\fill (9,6) circle (1.5mm);
\fill (10,6) circle (1.5mm);
\fill (11,6) circle (1.5mm);
\fill (12,6) circle (1.5mm);
\fill (13,6) circle (1.5mm);
\fill (14,6) circle (1.5mm);
\fill (15,6) circle (1.5mm);
\fill (16,6) circle (1.5mm);

\fill (0,9) circle (1.5mm);
\fill (2,9) circle (1.5mm);
\fill (4,9) circle (1.5mm);
\fill (6,9) circle (1.5mm);
\fill (8,9) circle (1.5mm);
\fill (10,9) circle (1.5mm);
\fill (12,9) circle (1.5mm);
\fill (14,9) circle (1.5mm);
\fill (16,9) circle (1.5mm);
\fill (18,9) circle (1.5mm);

\fill (9,12) circle (1.5mm);

\end{tikzpicture}
\caption{}
\label{G(RP2)}
\end{subfigure}
\caption{(a) A plane drawing of a minimal triangulation $K$ of $\mathbb{R}\mathbb{P}^2$, with outer edges appropriately identified. (b) The graph $G(K)$ obtained from this triangulation using the approach of Kaneta and Yoshinaga.}
\label{Fig:RP2TriangulationAndGraph}
\end{figure}
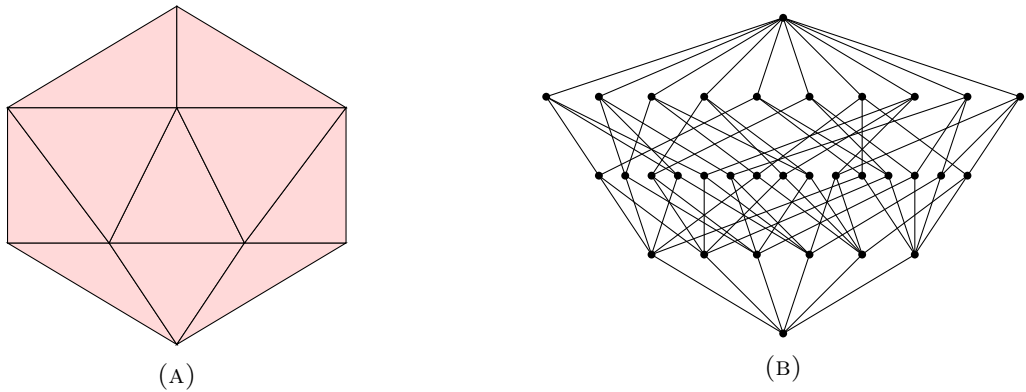
\end{exa}

We now extend this result of Kaneta and Yoshinaga in two directions. First, we show there exist graphs with torsion of order two in magnitude homology in bigradings other than $(3,4)$. Then, we use lens spaces to show that there are graphs with torsion of a given order in magnitude homology.

\begin{thm}\label{RPKtwotorsion}
For any odd integer $k \ge 3$, there is a graph $G$ such that $\textnormal{MH}_{k,k+1}(G)$ contains a subgroup isomorphic to $\Z_2$.
\end{thm}

\begin{proof}
$\widetilde{H}_{k-2}(\R\mathbb{P}^{k-1}) \cong \Z_2$ and for a triangulation $K$ of   \, $\mathbb{R}\mathbb{P}^{k-1}$, Theorem \ref{YKembeddinglemma} gives an embedding
$\widetilde{H}_{k-2}(\mathbb{R}\mathbb{P}^{k-1}) \hookrightarrow \textnormal{MH}_{k,k+1}(G(K))$.
\end{proof}

For coprime integers $p$ and $q$, the lens space $L(p,q)$ is a three-dimensional triangulable manifold. Since the fundamental group (hence first homology group) of $L(p,q)$ is isomorphic to $\Z_p$, it follows by Kaneta and Yoshinaga's embedding Theorem \ref{YKembeddinglemma} that torsion of a given order can show up in the magnitude homology of a graph. However, this approach would only guarantee the existence of graphs with torsion in polynomial degree $5$. Using generalized lens spaces, we can prove a more general result.

\begin{defn}[Generalized lens space, \cite{li2007fundamental}]
Let $S^{2n+1} = \{ (z_0,z_1,\ldots,z_n) \in \C^{n+1} \ : \ \sum_{i=0}^n |z_i|^2 = 1\}$ be the unit sphere in $C^{n+1}$. Let $p,q_1,q_2,\ldots,q_n$ be integers with $\textnormal{gcd}(p,q_i)=1$ for each $1 \le i \le n$. Consider the action of $\Z_p$ on $S^{2n+1}$ defined for each $g \in \Z_p$ by
\[ g \cdot (z_0,z_1,\ldots,z_n) = \left( z_0e^{\frac{2\pi i g}{p}},z_1e^{\frac{2\pi i g q_1}{p}},z_2e^{\frac{2\pi i g q_2}{p}},\ldots,z_ne^{\frac{2\pi i g q_n}{p}} \right). \]
The generalized lens space $L(p,q_1,q_2,\ldots,q_n)$ is the quotient space $S^{2n+1}/\Z_p$.
\end{defn}

The generalized lens space $L(p,q_1,q_2,\ldots,q_n)$ is  a triangulable $(2n+1)$-dimensional manifold with fundamental group isomorphic to $\Z_p$ \cite{li2007fundamental,rubinstein2018generalized}.

\begin{thm}\label{ptorsioninMH}
For each prime $p$ and positive integer $r$, there is a graph with $\Z_{p^r}$ torsion in its magnitude homology. More specifically, for integers $n,r \ge 1$ and each prime $p$, there is a graph $G$ such that $\textnormal{MH}_{3,2n+3}(G)$ contains $\Z_{p^r}$ torsion.
\end{thm}

\begin{proof}
Let $G(K)$ the graph obtained from a triangulation $K$ of the generalized lens space $L(p^r, q_1,q_2,\ldots,q_n)$. By Theorem \ref{YKembeddinglemma}, there is an embedding
\[ \Z_{p^r} \cong \pi_1(L(p^r, q_1,q_2,\ldots,q_n)) \cong \widetilde{H}_1(L(p^r, q_1,q_2,\ldots,q_n)) \hookrightarrow \textnormal{MH}_{3,2n+3}(G(K)). \]
\end{proof}

\begin{defn}
Let $r \ge 1$ be a positive integer and $G$ be a graph. The $r^{th}$ diagonal of the magnitude homology of $G$ is the sequence of groups $(\textnormal{MH}_{\ell-r+1, \ell}(G))_{\ell \ge r-1}$. We refer to the $1^{st}$ diagonal as the \textit{main} diagonal.
\end{defn}

Note that the proof of Theorem \ref{RPKtwotorsion} shows there can be torsion of order two in any bigrading along the second diagonal of magnitude homology, while Theorem \ref{ptorsioninMH} shows that, for any odd integer $r \ge 3$, torsion of arbitrary prime order can be found in homological degree $3$ of the $r^{th}$ diagonal.

So far we have constructed a single graph with torsion of a desired order. In the remainder of this section we provide affirmative answers  the following questions:
\begin{itemize}
    \item[A)] Is there a graph with torsion in magnitude homology which is not obtained from a triangulation via the method of Kaneta and Yoshinaga?
  \item[B)] Can we produce entire families of graphs with torsion of a given prime order?
\end{itemize}

\begin{prop}
There is a graph $G$, not obtained from a triangulation via the Kaneta-Yoshinaga construction, with torsion of order two in magnitude homology.
\end{prop}

\begin{proof}
The program rational\_graph\_homology\_arxiv.py was uploaded along with Hepworth and Willerton's paper \cite{hepworth2017categorifying}, and can be used to calculate magnitude homology groups over $\Q$ and over finite fields $\Z_p$. Calculations using this program show that removing a single edge from the graph of Figure \ref{G(RP2)} produces a graph whose magnitude homology has the same rank over $\Q$ and over $\Z_2$ coefficients. Consequently, removing a single edge destroys the property leading to $\Z_2$ torsion in the magnitude homology of this graph. On the other hand, adding a single edge to the same graph sometimes vanishes the $\Z_2$ torsion and sometimes does not. For instance, the graph in Figure \ref{NotYKGraph} contains $\Z_2$ torsion.
\end{proof}

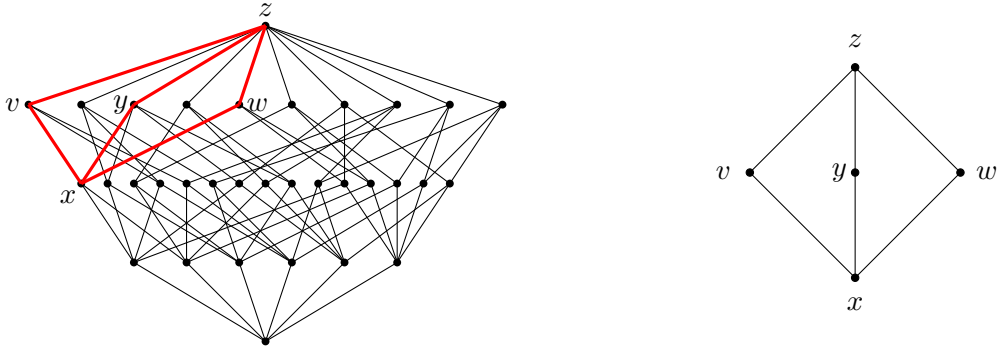
\begin{figure}[h]
\begin{subfigure}{0.48\textwidth}
\begin{center}
\begin{tikzpicture}[scale=.35]
%\draw[help lines] (0,0) grid (18,12);
\fill (9,0) circle (1.5mm);

\draw (9,0)--(4,3);
\draw (9,0)--(6,3);
\draw (9,0)--(8,3);
\draw (9,0)--(10,3);
\draw (9,0)--(12,3);
\draw (9,0)--(14,3);

\fill (4,3) circle (1.5mm);
\fill (6,3) circle (1.5mm);
\fill (8,3) circle (1.5mm);
\fill (10,3) circle (1.5mm);
\fill (12,3) circle (1.5mm);
\fill (14,3) circle (1.5mm);

\draw (2,6)--(4,3);
\draw (3,6)--(4,3);
\draw (5,6)--(4,3);
\draw (8,6)--(4,3);
\draw (12,6)--(4,3);

\draw (4,6)--(6,3);
\draw (6,6)--(6,3);
\draw (9,6)--(6,3);
\draw (13,6)--(6,3);
\draw (2,6)--(6,3);

\draw (3,6)--(8,3);
\draw (4,6)--(8,3);
\draw (7,6)--(8,3);
\draw (10,6)--(8,3);
\draw (14,6)--(8,3);

\draw (5,6)--(10,3);
\draw (6,6)--(10,3);
\draw (7,6)--(10,3);
\draw (11,6)--(10,3);
\draw (15,6)--(10,3);

\draw (8,6)--(12,3);
\draw (9,6)--(12,3);
\draw (10,6)--(12,3);
\draw (11,6)--(12,3);
\draw (16,6)--(12,3);

\draw (16,6)--(14,3);
\draw (15,6)--(14,3);
\draw (14,6)--(14,3);
\draw (13,6)--(14,3);
\draw (12,6)--(14,3);

\draw (0,9)--(2,6);
\draw (8,9)--(2,6);

\draw (2,9)--(3,6);
\draw (4,9)--(3,6);

\draw (6,9)--(4,6);
\draw (10,9)--(4,6);

\draw (0,9)--(5,6);
\draw (2,9)--(5,6);

\draw (0,9)--(6,6);
\draw (14,9)--(6,6);

\draw (2,9)--(7,6);
\draw (16,9)--(7,6);

\draw (4,9)--(8,6);
\draw (12,9)--(8,6);

\draw (6,9)--(9,6);
\draw (14,9)--(9,6);

\draw (4,9)--(10,6);
\draw (6,9)--(10,6);

\draw (14,9)--(11,6);
\draw (18,9)--(11,6);

\draw (8,9)--(12,6);
\draw (12,9)--(12,6);

\draw (8,9)--(13,6);
\draw (10,9)--(13,6);

\draw (10,9)--(14,6);
\draw (16,9)--(14,6);

\draw (16,9)--(15,6);
\draw (18,9)--(15,6);

\draw (12,9)--(16,6);
\draw (18,9)--(16,6);

\draw (0,9)--(9,12);
\draw (2,9)--(9,12);
\draw (4,9)--(9,12);
\draw (6,9)--(9,12);
\draw (8,9)--(9,12);
\draw (10,9)--(9,12);
\draw (12,9)--(9,12);
\draw (14,9)--(9,12);
\draw (16,9)--(9,12);
\draw (18,9)--(9,12);

\fill (2,6) circle (1.5mm);
\fill (3,6) circle (1.5mm);
\fill (4,6) circle (1.5mm);
\fill (5,6) circle (1.5mm);
\fill (6,6) circle (1.5mm);
\fill (7,6) circle (1.5mm);
\fill (8,6) circle (1.5mm);
\fill (9,6) circle (1.5mm);
\fill (10,6) circle (1.5mm);
\fill (11,6) circle (1.5mm);
\fill (12,6) circle (1.5mm);
\fill (13,6) circle (1.5mm);
\fill (14,6) circle (1.5mm);
\fill (15,6) circle (1.5mm);
\fill (16,6) circle (1.5mm);

\fill (0,9) circle (1.5mm);
\fill (2,9) circle (1.5mm);
\fill (4,9) circle (1.5mm);
\fill (6,9) circle (1.5mm);
\fill (8,9) circle (1.5mm);
\fill (10,9) circle (1.5mm);
\fill (12,9) circle (1.5mm);
\fill (14,9) circle (1.5mm);
\fill (16,9) circle (1.5mm);
\fill (18,9) circle (1.5mm);

\fill (9,12) circle (1.5mm);

%\draw () to [] ();

%%%%%% Bolded edges

\node at (-.6,9){$v$};
\node at (8.7,9){$w$};
\node at (1.5,5.5){$x$};
\node at (3.5,9){$y$};
\node at (9,12.6){$z$};
\draw[red,very thick] (2,6)--(4,9);
\draw[red, very thick] (4,9)--(9,12);
\draw[red,very thick] (2,6)--(0,9);
\draw[red,very thick] (0,9)--(9,12);
\draw[red, very thick] (2,6)--(8,9);
\draw[red, very thick] (8,9)--(9,12);

\end{tikzpicture}
\end{center}
\label{nonYKgraphwithtorsion}
\end{subfigure}
\begin{subfigure}{0.48\textwidth}
\begin{center}
\begin{tikzpicture}[scale=.7]

\draw[fill=black] (2,0) circle (.7mm);
\node at (2,-.5){$x$};

\draw[fill=black] (4,2) circle (.7mm);
\node at (4.5,2){$w$};

\draw[fill=black] (0,2) circle (.7mm);
\node at (-.5,2){$v$};

\draw[fill=black] (2,4) circle (.7mm);
\node at (2,4.5){$z$};

\draw[fill=black] (2,2) circle (.7mm);
\node at (1.7,2){$y$};

\draw (2,0)--(2,2);
\draw (2,2)--(2,4);
\draw (2,0)--(0,2);
\draw (2,0)--(4,2);
\draw (0,2)--(2,4);
\draw (4,2)--(2,4);

\end{tikzpicture}
\end{center}
\end{subfigure}
\caption{A graph obtained from that in Figure \ref{G(RP2)} by adding the single edge $\{x,y\}$. This graph cannot be obtained via the Kaneta-Yoshinaga construction due to the presence of the forbidden 5-vertex subgraph with vertices $v,w,x,y$ and $z$ (red).}
\label{NotYKGraph}
\end{figure}

\subsection{Infinite families of graphs with $\Z_{p^m}$ torsion in magnitude homology}\label{InfiniteFamiliesWithPTorsion}

In this section we show that there are infinitely many graphs containing torsion of a prescribed order.
%To achieve this, we will recall a set of moves on triangulations, called Pachner moves, that leave the underlying manifold unchanged\textemdash thus giving rise to new graphs that still have a desired prime order of torsion in magnitude homology.

\begin{defn}[\cite{Pachner1991Moves}]
Let $K$ be a triangulation of an $m$-manifold $M$. Let $A \subset K$ be a subcomplex of dimension $m$, and let $\varphi : A \rightarrow A' \subset \partial \Delta^{m+1}$ be a simplicial isomorphism. The \textit{Pachner move} associated to the triple $(K,A,\varphi)$ is the adjunction space
\[
	P_{\varphi}K := (K-A) 		\sqcup_{\varphi} (\partial\Delta^{m+1}-A')
\]
\end{defn}

In dimension $1$, Pachner moves consist of either subdividing an edge into two edges, or the reverse. Pachner moves in dimension $2$ are illustrated in Figure \ref{Fig: PachnerMoves}.

\begin{figure}[H]
\begin{center}
\begin{tikzpicture}[scale=1]
%\draw[help lines] (0,0) grid (16,6);

%%%%%%%%% n=2 Pachner Moves of type 1

% left triangle
\filldraw[draw=black, fill=red!25] (0,3.55)--(2,3.55)--(1,5.28)--cycle;

% right triangle
\filldraw[draw=black, fill=red!25] (4.5,3.55)--(6.5,3.55)--(5.5,5.28)--cycle;
\draw (5.5,4.15)--(4.5,3.55);
\draw (5.5,4.15)--(6.5,3.55);
\draw (5.5,4.15)--(5.5,5.28);

% squiggly arrows

\draw[->,decorate,decoration={snake}] (2.3,4.8)--(4.2,4.8) node[midway,above] {$P_1$};
\draw[->,decorate,decoration={snake}] (4.2,4.2)--(2.3,4.2);
\node at (3.35,3.6){$P_1^{-1}$};

% dividing line

\draw[dashed] (7.5,3.2)--(7.5,5.8);

% n=2 Pachner Moves of type 2

\filldraw[draw=black, fill=red!25] (8.5,3.55)--(10.23,3.55)--(10.23,5.28)--(8.5,5.28)--cycle;
\draw (8.5,3.55)--(10.23,5.28);

\filldraw[draw=black, fill=red!25] (12.93,3.55)--(14.7,3.55)--(14.7,5.28)--(12.93,5.28)--cycle;
\draw (14.7,3.55)--(12.93,5.28);

\draw[->,decorate,decoration={snake}] (10.6,4.8)--(12.5,4.8) node[midway,above] {$P_2$};
\draw[->,decorate,decoration={snake}] (12.5,4.2)--(10.6,4.2);
\node at (11.65,3.6){$P_2^{-1}$};
\end{tikzpicture}
\caption{Pachner moves on simplicial complexes of dimension $2$.}
\label{Fig: PachnerMoves}
\end{center}
\end{figure}
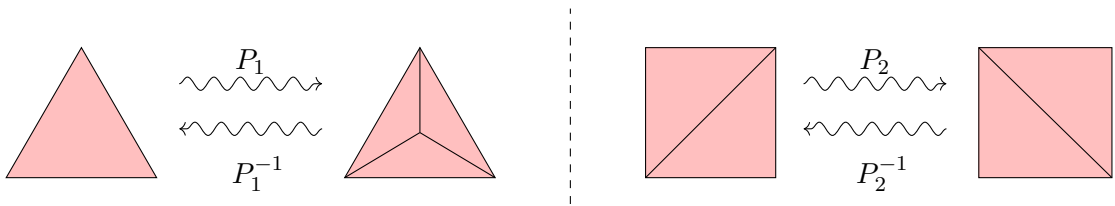

\vspace{-7mm}

Pachner moves send simplicial complexes to simplicial complexes, and preserve the underlying manifold in the sense that applying a Pachner move leaves the homeomorphism class of the geometric realization unchanged \cite{Pachner1991Moves}. Applying a Pachner move to a simplicial complex gives rise to a corresponding change on the level of graphs: $G(K) \mapsto G(P_{\varphi}K)$. As a consequence of Theorem \ref{YKembeddinglemma} we have the following.

\begin{cor}\label{PachnerInvariance}
Let $K$ and $K'$ be triangulations of a manifold $M$ related by a finite sequence of Pachner moves. For each $k \ge 1$, both $\textnormal{MH}_{k,*}(G(K))$ and $\textnormal{MH}_{k,*}(G(K'))$ contain a subgroup isomorphic to $\widetilde{H}_{k-2}(M)$.
\end{cor}

\begin{thm}\label{InfinitelyMany2Torsion}
Let $k \ge 3$ be an integer. There exist infinitely many distinct classes of graphs whose magnitude homology contains $\Z_2$ torsion in bigrading $(k,k+1)$.
\end{thm}

\begin{proof}
Let $K$ be a triangulation of $\mathbb{R}\mathbb{P}^{k-1}$ and define a sequence of graphs $(G_r)_{r \in \mathbb{N}}$ as follows. Let $K_1=K$ and for $r \ge 1$, define $K_{r+1} = P_{\varphi_r}K_r$ where $\varphi_r : A_r \rightarrow \partial \Delta^k$ is a simplicial isomorphism and $A_r$ is a $(k-1)$-simplex of $K_r$. Now set $G_r=G(K_r)$. These graphs are mutually distinct because, for example, the $\varphi_r$ have been chosen so that the number of simplices in $K_{r+1}$ is strictly greater than in $K_r$, and correspondingly we have $|V(G(K_{r+1})| > |V(G(K_r))|$. For each $r \ge 1$, $\textnormal{MH}_{k,k+1}(G_r)$ contains a subgroup isomorphic to $\Z_2$ by Theorem \ref{YKembeddinglemma} and Corollary \ref{PachnerInvariance}.
\end{proof}

\begin{thm}\label{InfinitelyManypTorsion}
Let $p$ be a prime and $n,m \ge 1$ integers. There exist infinitely many distinct isomorphism classes of graphs whose magnitude homology contains $\Z_{p^m}$ torsion in bigrading $(3,2n+3)$. 
\end{thm}

\begin{proof}
Let $p$ be a prime and $q_1,q_2,\ldots,q_n$ be integers coprime to $p^m$. Let $K$ be a triangulation of the lens space $L(p^m,q_1,q_2,\ldots,q_n)$. Set $K_1=K$ and for $r \ge 1$ define $K_{r+1}=P_{\varphi_r}K_r$ where $\varphi_r : A_r \rightarrow \partial \Delta^{2n+2}$ is a simplicial isomorphism and $A_r$ is a $(2n+1)$-simplex of $K_r$. By a similar argument as in the proof of Theorem \ref{InfinitelyMany2Torsion}, the graphs $G_r=G(K_r)$ are mutually distinct, and Theorem \ref{YKembeddinglemma} and Corollary \ref{PachnerInvariance} imply that $\textnormal{MH}_{3,2n+3}(G_r)$ contains a subgroup isomorphic to $\Z_{p^m}$ for each $r \ge 1$.
\end{proof}

Another way to produce families of graphs with a given prime order torsion was suggested to me by a member of my committee, Dr. Tye Lidman, to whom I am very grateful for the suggestion. It goes as follows. Let $p$ be a prime and for each $n \ge 1$, let $K_n$ be a triangulation of the lens space $L(p^n,1)$. By Theorem \ref{YKembeddinglemma}, $\textnormal{MH}(G(K_n))$ has a subgroup isomorphic to $\Z_{p^n}$. Dr. Lidman also pointed out that we can take any finitely generated abelian group and realize it as a subgroup of the singular homology of a topological space. Furthermore, we can always choose such a space that is \textit{triangulable}. As a consequence of the embedding of Theorem \ref{YKembeddinglemma}, we thus have the following.

\begin{thm}\label{FGAbGrp}
Let $M$ be any finitely generated finite abelian group. Then, there exists a graph $G$ whose magnitude homology $\textnormal{MH}(G)$ contains a subgroup isomorphic to $M$.
\end{thm}

\begin{proof}
By the fundamental theorem of finitely generated abelian groups, $M$ is of the form
\[
M \cong \Z^r \oplus \Z_{p_1^{r_1}} \oplus \Z_{p_2^{r_2}} \oplus \cdots \oplus \Z_{p_m^{r_m}}
\]
for some integers $r, r_i$ and $m$, and primes $p_i$ (not necessarily distinct). The lens spaces $L(p^{r_i},1)$ have first homology groups $H_1(L(p^{r_i},1)) \cong \Z_{p^{r_i}}$. For manifolds $X$ and $Y$ of the same dimension, the homology groups satisfy $H_k(X \# Y) \cong H_k(X) \oplus H_l(Y)$ for each $k \ge 0$, where $\#$ denotes the connected sum. This identity extends by induction to $m$-fold connected sums. Consequently,  $$H_1(L(p^{r_1},1) \ \# L(p^{r_2},1) \ \# \cdots \# \ L(p^{r_m},1)) \cong \Z_{p_1^{r_1}} \oplus \Z_{p_2^{r_2}} \oplus \cdots \oplus \Z_{p_m^{r_m}}.$$ Let $K$ be a triangulation of the 3-manifold $L(p^{r_1},1) \ \# L(p^{r_2},1) \ \# \cdots \# \ L(p^{r_m},1)$. Then, by Theorem \ref{YKembeddinglemma}, the graph $G(K)$ obtained via the Kaneta-Yoshinaga construction has magnitude homology with a subgroup isomorphic to $\Z_{p_1^{r_1}} \oplus \Z_{p_2^{r_2}} \oplus \cdots \oplus \Z_{p_m^{r_m}}$.  For any graph $G$ with at least one edge, $\rank(\textnormal{MH}_{k,k}(G)) \ge 1$ for every $k \ge 0$. Consequently, $G(K)$ also has a subgroup isomorphic to $\Z^r$ for every integer $r$. Consequently, the magnitude homology of $G(K)$ has a subgroup isomorphic to $M$.
\end{proof}

%\subsection{A bound on types of torsion in magnitude homology}\label{TorsionAndGeodesics}

%In this section, we introduce a bound on possible types of torsion, show this bound is sharp, and see how the presence of torsion in magnitude homology groups relates to counting geodesics in graphs viewed as metric spaces.

\section{Magnitude homology of outerplanar graphs and graphs with no $3$- or $4$-cycles}\label{ComputationsAndMainDiagonal}

In this section we use Gu's computations of the magnitude homology groups of cycle graphs \cite{gu2018graph} to compute the magnitude homology of a family of outerplanar graphs, compute the main diagonal of all graphs with no cycles of length $3$ or $4$, and put forth several conjectures regarding the main diagonal for other families of graphs based on calculations performed using Python.

\subsection{Cycle graphs and graphs without $3$- or $4$-cycles}

\begin{prop}\label{C3homology}
The magnitude homology of the cycle graph $C_3$ is torsion-free, supported on the main diagonal, and satisfies $\textnormal{MH}_{k,k}(C_3) \cong \Z^{3 \cdot 2^{k}}$.

\begin{comment} 
torsion-free group with

\[
\rank (\textnormal{MH}_{k,\ell}(C_3)) =
\left\{ \begin{array}{cc}
3 \cdot 2^k & k=\ell\ge0\\
0 & k \neq \ell
\end{array}
\right.
\]
\end{comment}

\end{prop}

\begin{proof}
	Consider a generator $(v_0, v_1, \ldots, v_k)$ of $\MC_{k,k}(C_3)$. Since $d(v_{j-1},v_{j+1}) \le 1$ while $d(v_{j-1},v_{j}) + d(v_{j},v_{j+1}) = 2$, it follows that $\partial(v_0, v_1, \ldots, v_i) = 0$ for every generator.
	Hence, $\textnormal{MH}_{k,k}(C_3) \cong \MC_{k,k}(C_3)$. There are three choices of initial vertex $v_0$ and two choices for each of the subsequent vertices $v_1,v_2,\ldots,v_k$, giving a total of $3 \cdot 2^k$ generators. For the last statement, simply note that for $l \neq k$ we have  $l(v_0,v_1,\ldots,v_k) = k \neq l$, so that $\MC_{k,\ell}(C_3) = 0$ for $l \neq k$.
\end{proof}

Using the K\"{u}nneth sequence, Hepworth and Willerton computed the magnitude homology of the cycle graph $C_4 = K_2 \Box K_2$ as follows.

\begin{prop}[\cite{hepworth2017categorifying}]
The magnitude homology of the cycle graph $C_4$ is torsion-free, supported on the main diagonal and satisfies $\textnormal{MH}_{k,k}(C_4) \cong \Z^{4 + 4k}$.

\begin{comment}
abelian group with
\[
\rank (\textnormal{MH}_{k,\ell}(C_4)) =
\left\{ \begin{array}{cc}
4+4k & k=\ell\ge0\\
0 & k \neq \ell
\end{array}
\right.
\]
\end{comment}

\end{prop}

\begin{thm}\label{nosquaresortriangles}
Let $G$ be a graph with vertex set $V$ and edge set $E$. If $G$ has no $3$- or $4$-cycles, then the first diagonal in the magnitude homology of $G$ is torsion-free and satisfies

\[
\textnormal{MH}_{k,k}(G) \cong
\left\{ \begin{array}{cc}
\Z^{|V|} & k=0,\\
\Z^{2|E|} & k>0.
\end{array}
\right.
\]

\end{thm}

\begin{proof}
Note that $\MC_{k+1,k}(G) = 0$, so $\textnormal{MH}_{k,k}(G)$ is the kernel of the map $\partial : \MC_{k,k}(G) \longrightarrow \MC_{k-1,k}(G)$. For $k=0$, this kernel is generated by the vertices of $V$. For $k=1$, there are $2|E|$ generating tuples $(x_0,x_1)$ for $\MC_{k,k}(G)$, and each has differential zero. For $k \ge 2$, assume without loss of generality that $k$ is odd, and consider the set
\[
	B = \{ (v,w,v,w,\ldots,w), \ (w,v,w,v,\ldots,v) \ | \ \{v,w\} \in E \}.
\]
Since $|B| = 2|E|$, it suffices to show that $B$ forms a basis for the kernel of $\partial : \MC_{k,k}(G) \longrightarrow \MC_{k-1,k}(G)$. Indeed, $\MC_{k,k}(G)$ is, by definition, generated by tuples $(x_0,x_1,\ldots,x_k)$ with $d(x_i,x_{i+1})=1$ for each $0 \le i \le k-1$. For such a tuple to lie in the kernel of $\partial$, it must satisfy $\partial(x_0,x_1,\ldots,x_k)=0$. This happens if and only iff $\partial_i(x_0,x_1,\ldots,x_k)=0$ for each $1 \le i \le k-1$. If there is an index $i$ with $x_i \neq x_{i+2}$, then $\partial_i(x_0,x_1,\ldots,x_k) = 0$ and $x_i, x_{i+1}$ and $x_{i+2}$ form a 3-cycle in $G$. But $G$ has no 3-cycles, so no such generators of $\MC_{k,k}(G)$ are in the kernel of $\partial$. It remains to show that no linear combination of generating tuples lies in the kernel of $\partial$. Let $c$ be a linear combination of generating tuples for $\MC_{k,k}(G)$, and assume without loss of generality that no tuple in $c$ belongs to the set $B$. Let $(x_0,x_1,\ldots,x_k) \in c$. For some $1 \le i \le k-1$, we have $d(x_{i-1},x_{i+1})=2$. Since $\partial(c)=0$, the tuple $(x_0,\ldots,x_{i-1},x_{i+1},\ldots,x_k) \in \partial(c)$ must cancel with $\partial_i(x_0,x_1,\ldots,x_{i-1},y,x_{i+1},\ldots,x_k) \in c-(x_0,x_1,\ldots,x_k)$. Then, $x_{i-1},x_i,x_{i+1}$ and $y$ form a 4-cycle in $G$.  But $G$ has no $4$-cycles.

\end{proof}

\subsection{Outerplanar Graphs}

\begin{comment}
In \cite{hepworth2017categorifying}, Hepworth and Willerton compute explicitly the magnitude homology groups of trees.

\begin{prop}
Let $T$ be a tree with vertex set $V$ and edge set $E$. $\textnormal{MH}_{*,*}(T)$ is torsion-free with

\[
\rank (\textnormal{MH}_{k,\ell}(T)) =
\left\{ \begin{array}{cc}
|V(T)| & k=\ell=0\\
2|E(T)| & k=\ell \ge 1\\
0 & k \neq \ell
\end{array}
\right.
\]
\end{prop}
\end{comment}

\begin{comment}
\begin{thm}
For $n \ge 5$, the magnitude homology of the cycle graph $C_n$ satisfies 
\begin{align*}
\textnormal{MH}_{i,i}(C_n) \cong
\left\{
\begin{array}{rr}
\Z^n & \;\;\; \text{if} \;\;\; i=0\\
\Z^{2n} & \;\;\; \text{if} \;\;\; i \ge 1
\end{array}
\right.
\end{align*}
\end{thm}

Observe that the preceding proof relied only on the fact that $C_n$ has no subgraphs to $C_3$ or $C_4$. Consequently, we have the following more general result.
\begin{cor}\label{nosquaresortriangles}
If a graph $G$ contains no subgraphs isomorphic to $C_3$ or $C_4$, then the first diagonal in the magnitude homology of $G$ is given by
\begin{align*}
\textnormal{MH}_{i,i}(G) \cong
\left\{
\begin{array}{rr}
\Z^{|V(G)|} & \;\;\; \text{if} \;\;\; i=0\\
\Z^{2|E(G)|} & \;\;\; \text{if} \;\;\; i \ge 0
\end{array}
\right.
\end{align*}
\end{cor}
\end{comment}

This sole dependence of the main diagonal on the number of vertices and edges extends to certain types of outerplanar graphs.

\begin{defn}
A graph $G$ is \textit{outer planar} if it has a plane drawing with no crossings all of whose vertices lies on an outer face of $G$. Equivalently, a graph is outer planar if it can be constructed from a finite collection $\{ H_1, H_2, \ldots, H_t \}$ of copies of $K_2$ (the complete graph on two vertices i.e. a single edge) and cycle graphs $C_n$, by gluing along single vertices or edges as follows. $G = H_1 \star H_2 \star \cdots \star H_t$ where for each $1 \le s \le t$, $H_1 \star  H_2 \star \cdots \star H_{s-1} \star H_s$ is formed from $H_1 \star  H_2 \star \cdots \star H_{s-1}$ by identifying an edge/vertex of $H_s$ with an outer-face edge/vertex of $H_1 \star  H_2 \star \cdots \star H_{s-1}$. We refer to the subgraphs $H_s$ as the components of $G$.
\end{defn}

\begin{thm}\label{evenouterplanardiagonal}
Let $G$ be an outer planar graph with vertex set $V$ and edge set $E$ whose components are either $K_2$ or even cycles $C_n$ with $n \neq 4$. Then, the main diagonal of the magnitude homology of $G$ is torsion-free with
\begin{align*}
\rank(\textnormal{MH}_{k,k}(G)) =
\left\{
\begin{array}{ll}
|V| & k=0,\\
2|E| & k > 0.
\end{array}
\right.
\end{align*}
\end{thm}

The proof of this theorem  will involve an induction argument, the Mayer-Vietoris sequence, and the following theorem of Hepworth and Willerton.

\begin{thm}[\cite{hepworth2017categorifying}]\label{treehomology}
Let $T$ be a tree with vertex set $V$ and edge set $E$. $\textnormal{MH}_{*,*}(T)$ is torsion-free group with

\[
\rank (\textnormal{MH}_{k,\ell}(T)) =
\left\{ \begin{array}{ll}
|V(T)| & k=\ell=0,\\
2|E(T)| & k=\ell \ge 1,\\
0 & k \neq \ell.
\end{array}
\right.
\]
\end{thm}

\begin{proof}[Proof of Theorem \ref{evenouterplanardiagonal}]
If $G$ is a tree, then the result follows from Theorem \ref{treehomology}. Otherwise, suppose $G = H_1 \star  H_2 \star \cdots \star H_t$ and let $X = H_1 \star  H_2 \star \cdots \star H_{t-1}$ and $Y = H_t$, where $H_t = C_n$ for some even integer $n\ge 6$. Then, $(G;X,Y)$ is a projecting decomposition. By the Mayer-Vietoris theorem, there is an isomorphism $\textnormal{MH}_{*,*}(G) \oplus \textnormal{MH}_{*,*}(X \cap Y) \cong \textnormal{MH}_{*,*}(X) \oplus \textnormal{MH}_{*,*}(Y)$. In particular,
\begin{equation}\label{MVrankequation}
\rank(\textnormal{MH}_{i,i}(G)) = \rank(\textnormal{MH}_{i,i}(X)) + \rank(\textnormal{MH}_{i,i}(Y)) - \rank(\textnormal{MH}_{i,i}(X \cap Y)).
\end{equation}

Since $X \cap Y$ is a tree, its magnitude homology is torsion-free with
\begin{align*}
\rank(\textnormal{MH}_{k,k}(X \cap Y)) =
\left\{
\begin{array}{ll}
|V(X \cap Y)| & k=0,\\
2|E(X \cap Y)| & k > 0.
\end{array}
\right.
\end{align*}

By Theorem \ref{nosquaresortriangles}, $\textnormal{MH}_{k,k}(Y)$ is torsion-free with
\begin{align*}
\rank(\textnormal{MH}_{k,k}(Y)) =
\left\{
\begin{array}{ll}
|V(Y)| & k=0,\\
2|E(Y)| & k>0.
\end{array}
\right.
\end{align*}

And by induction, the magnitude homology of $X$ is torsion-free with
\begin{align*}
\rank(\textnormal{MH}_{k,k}(X)) =
\left\{
\begin{array}{ll}
|V(X)| & k=0,\\
2|E(X)| & k>0.
\end{array}
\right.
\end{align*}

Equation (\ref{MVrankequation}) then gives
\begin{align*}
\rank((\textnormal{MH}_{k,k}(G)) &=
\left\{
\begin{array}{ll}
|V(X)| + |V(Y)| - |X \cap Y| & k=0,\\
2|E(X)| + 2|E(Y)| - 2|E(X \cap Y)| & k>0,\\
\end{array}
\right.\\
&= \left\{
\begin{array}{cc}
|V(G)| & k=0,\\
2|E(G)| & k>0.
\end{array}
\right.
\end{align*}
\end{proof}

In 2015, Hepworth and Willerton proposed the following recursive formula for the magnitude homology of the even cycle graphs based on experimental data \cite{hepworth2017categorifying}, which was subsequently proved in 2018 by Yuzhou Gu \cite{gu2018graph} using the tools of algebraic Morse theory.

\begin{thm}[\cite{gu2018graph}]
Fix an integer $m \ge 3$. The magnitude homology of the cycle graph $C_{2m}$ is described as follows.
\begin{enumerate}
    \item[(1)] All groups $\textnormal{MH}_{k,\ell}(C_{2m})$ are torsion-free.
    \item[(2)] Define a function $T : \Z \times \Z \rightarrow \Z$ as
    \begin{enumerate}
        \item $T(k,\ell) = 0$ if $k < 0$ or $\ell < 0$;
        \item $T(0,0) = 2m, T(1,1) = 4m$;
        \item $T(k,\ell) = \max\{T(k-1,\ell-1), T(k-2,\ell-m)\}$ for $(k,\ell) \neq (0,0), (1,1)$. 
    \end{enumerate}
    Then, $\rank(\textnormal{MH}_{k,\ell}(C_{2m})) = T(k,\ell)$.
\end{enumerate}
\end{thm}

As pointed out by Hepworth and Willerton \cite{hepworth2017categorifying}, the magnitude homology groups for the even cycle graphs are given by the following equivalent, but more explicit, formula.

\begin{obs}[\cite{hepworth2017categorifying}]
Fix an integer $m \ge 3$. The magnitude homology groups of the cycle graph are torsion-free and arrange themselves into diagonals. Let $T^m_{i,j}$ denote the rank of the magnitude homology group in the $j^{th}$ entry of the $i^{th}$ diagonal. That is, $T^m_{i,j} = \rank(\textnormal{MH}_{2(i-1)+(j-1),m(i-1)+(j-1)}(C_{2m}))$. Then, $T^m_{i,1}=2m$ for each $i \ge 1$, while $T^m_{i,j}=4m$ whenever $i \ge 1$ and $j \ge 2$. See Table \ref{C8Homology}.
\end{obs}

\begin{table}
    \begin{center}
    \begin{tabular}{ |c|c|c|c|c|c|c|c|c|c|c|c|c } 
    \hline
$\rk MH(C_8)$& 0 & 1 & 2 & 3 & 4 & 5 & 6 & 7 & 8 & 9 & 10\\
\hline
0& 8 & & & & & & & & & &\\ 
\hline
1& & 16 & & & & & & & & &\\
\hline
2& & & 16 & & & & & & & &\\
\hline
3& & & & 16 & & & & & & &\\
\hline
4& & & 8 & & 16 & & & & & &\\
\hline
5& & & & 16 & & 16 & & & & &\\
\hline
6& & & & & 16 & & 16 & & & &\\
\hline
7& & & & & & 16 & & 16 & & &\\
\hline
8& & & & & 8 & & 16 & & 16 & &\\
\hline
9& & & & & & 16 & & 16 & & 16 &\\
\hline
10& & & & & & & 16 & & 16 & & 16\\
\hline
\end{tabular}
    \end{center}
    \caption{The ranks of the torsion-free magnitude homology groups of the cycle graph $C_8$ \cite{hepworth2017categorifying}.}
    \label{C8Homology}
\end{table}

Applying the Mayer-Vietoris sequence (\ref{MVSequence}) and an analogous induction argument to that given in the proof of Theorem \ref{evenouterplanardiagonal} yields the following explicit formula for not only the first diagonal of magnitude homology, but all magnitude homology groups for a family of outerplanar graphs constructed from even cycles.

\begin{thm}\label{Evenouterplanarconjecture}
Fix a positive integer $m \ge 3$. Let $G$ be an outer planar graph with $S$ components $C_4$, and $R$ component cycles $C_{2m}$, constructed using edge-gluings only. The magnitude homology groups of $G$ are torsion-free and arrange themselves in diagonals: let $S^m_{i,j}$ denote the rank of the magnitude homology group of $G$ in the $j^{th}$ entry of the $i^{th}$ diagonal, that is, $S_{i,j}^m = \rank(\textnormal{MH}_{2(i-1) + (j-1), m(i-1) + (j-1)}(G))$. Then, the magnitude homology groups $\textnormal{MH}_{k,\ell}$ of $G$ are all trivial groups except for the groups the aforementioned diagonals, and these satisfy for $i > 1$
\begin{align*}
\rank(S_{1,j}^m) =
\left\{
\begin{array}{ll}
2mR + 4S - 2(R + S - 1) & j=1,\\
4mR +4jS - 2(R + S - 1) & j>1,
\end{array}
\right.
% =
% \left\{
% \begin{array}{rr}
% r|V(G)| - (r-1)|V(K_2)| & \;\;\; \text{if} \;\;\; j=1\\
% r2|E(G)| - (r-1)2|E(K_2)| & \;\;\; \text{if} \;\;\; j > 1
% \end{array}
% \right.
\end{align*}
\begin{align*}
\rank(S_{i,j}^m) =
\left\{
\begin{array}{ll}
2mR +4S - 2(R + S -1) & j=1,\\
4mR - 2(R + S - 1) & j>1.
\end{array}
\right.
% =
% \left\{
% \begin{array}{rr}
% 2nr - 2(r-1) & \;\;\; \text{if} \;\;\; j=1\\
% 4nr - 2(r-1) & \;\;\; \text{if} \;\;\; j > 1
% \end{array}
% \right.
\end{align*}

\end{thm}

\section{Conjectures and future directions}\label{lastsec}

As a relatively new invariant of graphs, not much is yet known about the relationship between structural properties of graphs and their magnitude homology groups. One such relationship investigated here was a connection between the building blocks (component cycles) of certain outer planar graphs and the ranks of magnitude homology groups. In Theorem \ref{Evenouterplanarconjecture}, we computed the magnitude homology groups of a family of outer planar graphs built from even cycle graphs. Despite the fact that Gu has also computed the magnitude homology groups of the odd cycle graphs, we cannot extend Theorem \ref{Evenouterplanarconjecture} result to the analogous family built instead from odd cycle graphs, or combinations of even and odd cycle graphs. This is due to the fact that the odd cycle graphs do not project onto their edges, and this means we cannot apply the Mayer-Vietoris sequence (Hepworth and Willerton showed that the the "projecting" condition in the statement of the Mayer-Vietoris sequence is strictly necessary \cite{hepworth2017categorifying}). Moving forward, we will attempt to extend Gu's algebraic Morse theory approach to compute the magnitude homology of the remaining outer planar graphs. We would then, for example, be in a better position to determine whether magnitude homology detects outer planarity.

In pursuit of other connections between graphical structure and magnitude homology, computations performed in Python are an excellent source of data from which we might derive hints. Below, we highlight a few such computations for members of some families of graphs, and offer conjectures for their magnitude homology groups. Many families investigated in this manner over the course of the research for this thesis pertained to graphs obtained by "gluing" cycle graphs of various sizes along edges, collections of edges, or vertices, for example the square polyominos of Definition \ref{SqPolyDefn}. We end by displaying some computations and accompanying conjectures for a small sampling of such families.

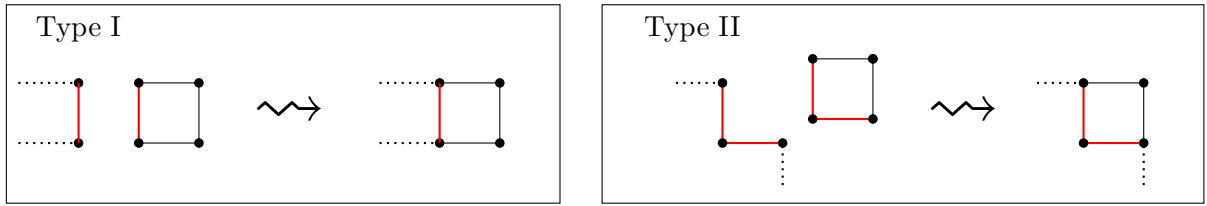
\begin{figure}[h]
\begin{center}
\begin{tikzpicture}[scale=.8]
%\draw[help lines] (0,0) grid (16,8);

%%%%%%%%%%%%%% Type I
\begin{scope}[xshift=-7.7cm]
\draw[dotted,thick] (0,5)--(1,5);
\draw[dotted,thick] (0,6)--(1,6);
\draw[fill=black] (1,5) circle (.7mm);
\draw[fill=black] (1,6) circle (.7mm);
\draw[red,thick] (1,5)--(1,6);

\draw[red,thick] (2,5)--(2,6);
\draw[fill=black] (2,5) circle (.7mm);
\draw[fill=black] (2,6) circle (.7mm);
\draw[fill=black] (3,5) circle (.7mm);
\draw[fill=black] (3,6) circle (.7mm);
\draw (2,5)--(3,5);
\draw (2,6)--(3,6);
\draw (3,5)--(3,6);

\node at (1,6.9){Type I};
\node[scale=2.5] at (4.5,5.5){$\rightsquigarrow$};
% \draw (0,6.6)--(.3,6.6);
% \draw (.3,6.6)--(.3,7.2);

\draw[dotted,thick] (6,5)--(7,5);
\draw[dotted,thick] (6,6)--(7,6);
\draw[fill=black] (7,5) circle (.7mm);
\draw[fill=black] (7,6) circle (.7mm);
\draw[red,thick] (7,5)--(7,6);
\draw (7,5)--(8,5);
\draw(8,5)--(8,6);
\draw(8,6)--(7,6);
\draw[fill=black] (8,5) circle (.7mm);
\draw[fill=black] (8,6) circle (.7mm);

\draw (-.2,4)--(9,4)--(9,7.3)--(-0.2,7.3)--cycle;
\end{scope}
%%%%%%%%%%%%%% Type II
\begin{scope}[xshift=3cm, yshift=4.5cm]

\node at (.5,2.4){Type II};
\begin{scope}[yshift=-0.5cm]
\draw[dotted,thick] (2,0.3)--(2,1);
\draw[dotted,thick] (1,2)--(0.2,2);
\draw[red,thick] (2,1)--(1,1);
\draw[red,thick] (1,1)--(1,2);

\draw[red,thick] (2.5,1.4)--(3.5,1.4);
\draw[red,thick] (2.5,1.4)--(2.5,2.4);
\draw (3.5,1.4)--(3.5,2.4);
\draw (8,1)--(8,2);
\draw (2.5,2.4)--(3.5,2.4);
\draw[fill=black] (3.5,1.4) circle (.7mm);

\draw[fill=black] (2,1) circle (.7mm);
\draw[fill=black] (1,1) circle (.7mm);
\draw[fill=black] (1,2) circle (.7mm);
\draw[fill=black] (3.5,2.4) circle (.7mm);
\draw[fill=black] (2.5,1.4) circle (.7mm);
\draw[fill=black] (2.5,2.4) circle (.7mm);
\draw[fill=black] (3.5,2.4) circle (.7mm);
\node[scale=2.5] at (5,1.5){$\rightsquigarrow$};

\draw[dotted,thick] (8,0.3)--(8,1);
\draw[dotted,thick] (7,2)--(6.2,2);
\draw[red,thick] (8,1)--(7,1);
\draw[red,thick] (7,1)--(7,2);
\draw[fill=black] (8,1) circle (.7mm);
\draw[fill=black] (7,1) circle (.7mm);
\draw[fill=black] (7,2) circle (.7mm);
%\draw (8,1)--(8,2);
\draw (8,2)--(7,2);
\draw[fill=black] (8,2) circle (.7mm);
\end{scope}
\draw (9,-.5)--(9,2.8)--(-1,2.8)--(-1,-.5)--cycle;
% \draw (-1,2.9)--(.4,2.9);
% \draw (.4,2.9)--(.4,3.5);
\end{scope}
\end{tikzpicture}
\end{center}
\caption{Gluings of Type I and Type II.}
\label{SquarePolyominosConstructions}
\end{figure}

\begin{defn}\label{SqPolyDefn}
The set $G^4_S$ of square polyominos on $S$ copies of $C_4$ is defined inductively as follows. $G^4_1 = C_4$ and each member of $G^4_{S+1}$ is obtained from a member of $G^4_S$ by gluing a copy of $C_4$ via a move of Type I or Type II, shown in Figure \ref{SquarePolyominosConstructions}.
\end{defn}

\begin{figure}[h]

\begin{center}
\begin{tikzpicture}[scale=1]
%\draw[help lines] (0,0) grid (14,3);

\node at (3.5,.5){\usebox\squaregraph};
\node at (4.5,.5){\usebox\squaregraph};
\node at (5.5,.5){\usebox\squaregraph};
\node at (5.5,1.5){\usebox\squaregraph};
\node at (4.5,1.5){\usebox\squaregraph};

\node at (9.5,.5){\usebox\squaregraph};
\node at (10.5,.5){\usebox\squaregraph};
\node at (11.5,.5){\usebox\squaregraph};
\node at (10.5,2.5){\usebox\squaregraph};
\node at (10.5,1.5){\usebox\squaregraph};

\end{tikzpicture}
\end{center}
\caption{Two polyominos $P_1$ and $P_2$ of type $G_4^5$.}
\label{TwoPolyominos}
\end{figure}
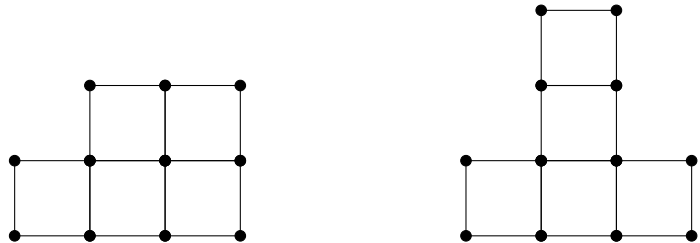

\begin{table}[H]
        \begin{minipage}{0.48\textwidth}
            \centering
            \begin{tabular}{|c|c|c|c|c|c|c|c|c| } 
            \hline
$P_1$& 0 & 1 & 2 & 3 & 4 & 5 & 6\\
\hline
0& 11 & & & & & &\\ 
\hline
1& & 30 & & & & &\\
\hline
2& & & 50 & & & &\\
\hline
3& & & & 70 & & &\\
\hline
4& & & & & 90 & &\\
\hline
5& & & & & & 110 &\\
\hline
6& & & & & & & 130\\
\hline
\end{tabular}
        \end{minipage}
        \hfill
        \begin{minipage}{0.48\textwidth}
            \centering
            \begin{tabular}{ |c|c|c|c|c|c|c|c|c| } 
            \hline
$P_2$& 0 & 1 & 2 & 3 & 4 & 5 & 6\\
\hline
0& 12 & & & & & &\\ 
\hline
1& & 32 & & & & &\\
\hline
2& & & 52 & & & &\\
\hline
3& & & & 72 & & &\\
\hline
4& & & & & 92 & &\\
\hline
5& & & & & & 112 &\\
\hline
6& & & & & & & 132\\
\hline
\end{tabular}
        \end{minipage}

\caption{The ranks of the magnitude homology computations for the two members of the set $G^4_5$ of polyominos given in Figure \ref{TwoPolyominos}.}
\label{fig:squarepolyominosexamples}
\end{table}

Based on the computations given in Table \ref{fig:squarepolyominosexamples} along with many others not displayed, we make the following conjecture.

\begin{conj}\label{Squarepolyominosconjecture} Let $S$ denote the number of  squares $C_4$ in the square polyomino $G^{4}_S.$ The main diagonal of the magnitude homology of a square polyomino is torsion-free and satisfies
\begin{align*}
\rank(\textnormal{MH}_{k,k}(G_S^4)) \cong
\left\{
\begin{array}{ll}
|V(G_S^4)| & k=0,\\
2|E(G_S^4)| + 4(i-1)S & k \ge 1.
\end{array}
\right.
\end{align*}
\end{conj}

% \[
% rk(\textnormal{MH}_{0,0}(G^4_S)) = |V| \qquad \text{and} \qquad rk(\textnormal{MH}_{i,i}(G^4_S)) = 2|E| + 4(i-1)S
% \]
In other words, magnitude homology is counting the number of vertices, edges and squares.

By Theorem \ref{evenouterplanardiagonal}, note that Conjecture \ref{Squarepolyominosconjecture} holds true for polyominos constructed using exclusively moves of type I. However, we cannot appeal to the Mayer-Vietoris sequence in any simple manner when moves of type II are used for the simple reason that $C_4$ does not project onto a pair of its adjacent edges.

Based on the computations in Table \ref{fig:squareplustwotriangles} and others, we posit the following.

\begin{conj}
The magnitude homology of a graph obtained by gluing two cycle graphs $C_3$ along single edges to a single cycle graph $C_4$ has diagonal magnitude homology provided those triangles are not attached to opposite sides of the $4$-cycle.
%See Figure \ref{TwoTrianglesOneSquare}.
\end{conj}

\begin{conj}
Gluing any number of cycle graphs $C_3$ along single edges to a cycle graph $C_4$ results in a graph with diagonal magnitude homology provided no pair of triangles is glued to opposite faces of the cycle graph $C_4$.
\end{conj}

\begin{figure}[H]

\begin{center}
\begin{tikzpicture}[scale=1]
%\draw[help lines] (0,0) grid (14,3);

\node at (1.5,.5){\usebox\squaregraph};
\draw[fill=black] (1.5,2) circle (.7mm);
\draw[fill=black] (3,.5) circle (.7mm);
\draw (1,1)--(1.5,2);
\draw (1.5,2)--(2,1);
\draw (2,0)--(3,.5);
\draw (3,.5)--(2,1);

\node at (6.5,.5){\usebox\squaregraph};
\draw[fill=black] (5,.5) circle (.7mm);
\draw[fill=black] (8,.5) circle (.7mm);
\draw (6,0)--(5,.5);
\draw (5,.5)--(6,1);
\draw (7,0)--(8,.5);
\draw (8,.5)--(7,1);

\end{tikzpicture}
\end{center}
\caption{Graphs $Sq_1$ and $Sq_2$ obtained by gluing two triangle graphs to a single square graph.}
\label{TwoTrianglesOneSquare}
\end{figure}
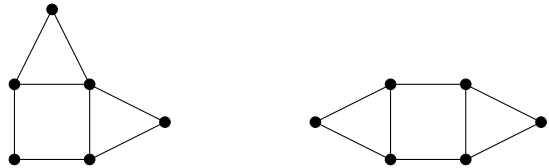

\begin{table}[H]
        \begin{minipage}{0.45\textwidth}
            \centering
            \begin{tabular}{ |c|c|c|c|c|c|c|c|c|c|c| } 
            \hline
$Sq_1$& 0 & 1 & 2 & 3 & 4 & 5 & 6 & 7\\
\hline
0& 6 & & & & & & &\\ 
\hline
1& & 16 & & & & & &\\
\hline
2& & & 32 & & & & &\\
\hline
3& & & & 60 & & & &\\
\hline
4& & & & & 112 & & &\\
\hline
5& & & & & & 212 & &\\
\hline
6& & & & & & & 408 &\\
\hline
7& & & & & & & & 796\\
\hline
\end{tabular}
        \end{minipage}
        \hfill
        \begin{minipage}{0.45\textwidth}
            \centering
            \begin{tabular}{ |c|c|c|c|c|c|c|c|c|c|c| } 
            \hline
$Sq_2$& 0 & 1 & 2 & 3 & 4 & 5 & 6 & 7\\
\hline
0& 6 & & & & & & &\\ 
\hline
1& & 16 & & & & & &\\
\hline
2& & & 32 & & & & &\\
\hline
3& & & 2 & 60 & & & &\\
\hline
4& & & & 12 & 112 & & &\\
\hline
5& & & & & 44 & 212 & &\\
\hline
6& & & & & 2 & 132 & 408 &\\
\hline
7& & & & & & 16 & 356 & 796\\
\hline
\end{tabular}
        \end{minipage}
\caption{The ranks of the magnitude homology of two graphs $Sq_1$ and $Sq_2$ given in Figure \ref{TwoTrianglesOneSquare}.}
\label{fig:squareplustwotriangles}
\end{table}

Based on the computations in Table \ref{fig:wheelsfiveandeight} we conjecture the following.

\begin{conj}
Wheel graphs $W_n$ on $n$ vertices have diagonal magnitude homology which is torsion-free and satisfies
\begin{align*}
\rank(\textnormal{MH}_{k,k}(W_n)) \cong
\left\{
\begin{array}{ll}
|V(W_n)| & k=0,\\
2|E(W_n)| \cdot 3^{k-1} & k \ge 1.
\end{array}
\right.
\end{align*}
\end{conj}

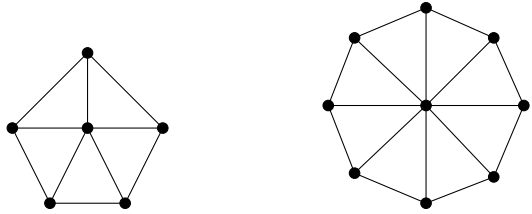
\begin{figure}[H]
\begin{center}
\begin{tikzpicture}[scale=1]
%\draw[help lines] (0,0) grid (16,6);

\draw[fill=black] (1,0) circle (.7mm);
\draw[fill=black] (2,0) circle (.7mm);
\draw[fill=black] (.5,1) circle (.7mm);
\draw[fill=black] (2.5,1) circle (.7mm);
\draw[fill=black] (1.5,2) circle (.7mm);
\draw[fill=black] (1.5,1) circle (.7mm);

\draw (1,0)--(2,0);
\draw (2,0)--(2.5,1);
\draw (2.5,1)--(1.5,2);
\draw (1.5,2)--(.5,1);
\draw (.5,1)--(1,0);

\draw (1.5,1)--(1,0);
\draw (1.5,1)--(2,0);
\draw (1.5,1)--(2.5,1);
\draw (1.5,1)--(1.5,2);
\draw (1.5,1)--(.5,1);

\draw[fill=black] (6,1.3) circle (.7mm);
\draw[fill=black] (6,0) circle (.7mm);
\draw[fill=black] (6,2.6) circle (.7mm);
\draw[fill=black] (7.3,1.3) circle (.7mm);
\draw[fill=black] (4.7,1.3) circle (.7mm);
\draw[fill=black] (6.9,2.2) circle (.7mm);
\draw[fill=black] (6.9,.35) circle (.7mm);
\draw[fill=black] (5.05,.4) circle (.7mm);
\draw[fill=black] (5.05,2.2) circle (.7mm);

\draw (6,0)--(6.9,.35);
\draw (6.9,.35)--(7.3,1.3);
\draw (7.3,1.3)--(6.9,2.2);
\draw (6.9,2.2)--(6,2.6);
\draw (6,2.6)--(5.05,2.2);
\draw (5.05,2.2)--(4.7,1.3);
\draw (4.7,1.3)--(5.05,.4);
\draw (5.05,.4)--(6,0);

\draw (6,1.3)--(6,0);
\draw (6,1.3)--(6.9,.35);
\draw (6,1.3)--(7.3,1.3);
\draw (6,1.3)--(6.9,2.2);
\draw (6,1.3)--(6,2.6);
\draw (6,1.3)--(5.05,2.2);
\draw (6,1.3)--(4.7,1.3);
\draw (6,1.3)--(5.05,.4);

\end{tikzpicture}
\end{center}
\caption{Wheel graphs $W_5$ and $W_8$.}
\end{figure}

\begin{table}[H]
        \begin{minipage}{0.45\textwidth}
            \centering
            \begin{tabular}{ |c|c|c|c|c|c|c|c|c| } 
            \hline
$W_5$& 0 & 1 & 2 & 3 & 4 & 5 & 6\\
\hline
0& 6 & & & & & &\\ 
\hline
1& & 20 & & & & &\\
\hline
2& & & 60 & & & &\\
\hline
3& & & & 180 & & &\\
\hline
4& & & & & 540 & &\\
\hline
5& & & & & & 1620 &\\
\hline
6& & & & & & & 4860\\
\hline
\end{tabular}
        \end{minipage}
        \hfill
        \begin{minipage}{0.45\textwidth}
            \centering
\begin{tabular}{|c|c|c|c|c|c|c|c|c|}
\hline
$W_6$& 0 & 1 & 2 & 3 & 4 & 5 & 6\\
\hline
0& 9 & & & & & &\\ 
\hline
1& & 32 & & & & &\\
\hline
2& & & 96 & & & &\\
\hline
3& & & & 288 & & &\\
\hline
4& & & & & 864 & &\\
\hline
5& & & & & & 2592 &\\
\hline
6& & & & & & & 7776\\
\hline
\end{tabular}
\end{minipage}
\caption{The ranks of the magnitude homology of the wheel graphs $W_5$ and $W_8$.}
\label{fig:wheelsfiveandeight}
\end{table}

During the course of research, we came to suspect a relationship between pairwise geodesic counts and permissible types of torsion in magnitude homology. By a geodesic in a graph $G$, we mean a $k$-path $(x_0,x_1,\ldots,x_k)$ in satisfying $\ell(x_0,x_1,\ldots,x_k) = d(x_0,x_k)$. In other words, a geodesic is a path of shortest length connecting two vertices. We suspect that the quantity $g(G) = \max_{x,y \in V{G}}g(x,y)$, where $g(x,y)$ is the number of geodesics connecting $x$ to $y$, plays a role in determining possible types of torsion in the magnitude homology of the graph $G$. Although this has not yet come to fruition, we will investigate this further in future work.

\clearpage
%\nocite{*}
\bibliographystyle {amsalpha}
\bibliography{bib.bib}

%\bibliographystyle{unsrt}
%\bibliographystyle{alphaurl}
%\bibliography{bib} 
\end{document}